\documentclass{article} 
\usepackage{framed,color,amssymb,graphicx,amsmath,amsthm,tikz,gensymb,tkz-euclide,dsfont,subfigure}

\usetikzlibrary{decorations.pathreplacing,calc}
\usetkzobj{all}

\newtheorem{theorem}{Theorem}
\newtheorem{claim}{Claim}
\newtheorem{lemma}{Lemma}
\theoremstyle{definition}
\newtheorem{problem}{Problem}
\newtheorem{definition}{Definition}

\newtheorem{corollary}{Corollary}

\title{Triangular Ramsey Numbers}
\author{Zachary Chaney\\
\and
Connor Mattes\\
\and
Jed Menard\\
\and Timothy Trujillo}

\begin{document}
\maketitle
\begin{abstract}
	The purpose of this paper is to introduce the idea of triangular Ramsey numbers and provide values as well as upper and lower bounds for them. To do this, the combinatorial game Mines is introduced; after some necessary theorems about triangular sets are proved. This game is easy enough that young children are able to play. The most basic variations of this game are analyzed and theorems about winning strategies and the existence of draws are proved. The game of Mines is then used to define triangular Ramsey numbers. Lower bounds are found for these triangular Ramsey numbers using the probabilistic method and the theorems about triangular sets. 
\end{abstract}

\section{Combinatorial Games}\label{Section1}
Combinatorial games make it possible to easily explain the underlying workings of combinatorial problems, which in turn help build a deeper understanding of combinatorics. In fact, some combinatorial games are so simple that they where invented as tools to be used in grade school classrooms to introduce students to logical reasoning and mathematical concepts. For example, the game of Tri was introduced by Haggard and Schonberger in \cite{haggard1977game} with the goal of ``developing logical skills of evaluating alternatives and their consequences." The game had the unintended learning outcome of developing the skill of visual disembedding, \emph{i.e.} the skill of picking out simple figures from a more complex image. Haggard and Schonberger in \cite{haggard1977game} point out that this ability has been linked to success in solving mathematical problems. In this paper we introduce a new game in the spirit of Haggard and Schonberger. In theory, these games are simple enough that they can be used to help develop mathematical problem solving skills in primary school students.

This paper focuses on a new combinatorial game called Mines which we use to introduce the notion of a triangular Ramsey number. The game is called Mines, because our original game boards were in the shape of a Reuleaux triangle, the logo of Colorado School of Mines (see Figure \ref{gameboard}). Our main results concern theorems about this game, such as the existence of a winner and the possibility of a winning strategy. The existence of triangular Ramsey numbers follows from the work of Dobrinen and Todorcevic in \cite{DandT}. The primary purpose of introducing the game of Mines is to provide a simplified presentation of the finite-dimensional Ramsey theory of the infinite-dimensional topological Ramsey space $\mathcal{R}_{1}$ introduced and studied by Dobrinen and Todorcevic in \cite{DandT}. The game provides a simplified method for defining triangular Ramsey numbers which are the direct analogue of the Ramsey numbers for the finite-dimensional Ramsey theory of $\mathcal{R}_{1}$. 

\begin{figure}[h]
\label{gameboard}
	\begin{center}
		\begin{tikzpicture}	
		\path[use as bounding box] (-2.5,-5) rectangle (8.5, 0);
		\begin{scope}	
		\draw (0,-0.6) node{} (-1.3,-2.1) node{} (1.3,-2.1) node{} (-2,-4.1) node{} (0,-4.5) node{} (2,-4.1) node{};
		\begin{scope}
		\foreach \x/\y in {1/-0.5, -1/-0.5, 0/1}
		{\clip (\x,\y) arc(120:180:5) arc(-120:-60:5) arc(0:60:5);}
		\fill[gray] (0,0) arc(120:180:5) arc(-120:-60:5) arc(0:60:5);
		\end{scope}
		\begin{pgfinterruptboundingbox}
		\path[clip,draw] (0,0) arc(120:180:5) arc(-120:-60:5) arc (0:60:5);
		\end{pgfinterruptboundingbox}
		\foreach \x/\y in {0/0, 0/1,1/-0.5,-1/-0.5}		
		{\draw (\x,\y) arc(120:180:5) arc(-120:-60:5) arc(0:60:5);}	
		\end{scope}
		
		\begin{scope}	
		\draw (6,-0.6) node{} (5.25,-1.25) node{} (6.75,-1.25) node{} (4.6,-2.1) node{} (6,-2.1) node{} (7.4,-2.1) node{} (4.2,-3.2) node{} (5.2,-3.4) node{} (6.8,-3.4) node{} (7.8,-3.2) node{} (4,-4.1) node{} (5,-4.4) node{} (6,-4.5) node{} (7,-4.4) node{} (8,-4.1) node{};
		\begin{scope}
		\foreach \x/\y in {5/-0.5, 6/2.05, 7.85/-1.1}
		{\clip (\x,\y) arc(120:180:5) arc(-120:-60:5) arc(0:60:5);}
		\fill[gray] (6,0) arc(120:180:5) arc(-120:-60:5) arc(0:60:5);
		\end{scope}
		\begin{scope}
		\foreach \x/\y in {6/1, 4.15/-1.1, 7.85/-1.1}
		{\clip (\x,\y) arc(120:180:5) arc(-120:-60:5) arc(0:60:5);}
		\fill[gray] (6,0) arc(120:180:5) arc(-120:-60:5) arc(0:60:5);
		\end{scope}
		\begin{scope}
		\foreach \x/\y in {7/-0.5, 6/2.05, 4.15/-1.1}
		{\clip (\x,\y) arc(120:180:5) arc(-120:-60:5) arc(0:60:5);}
		\fill[gray] (6,0) arc(120:180:5) arc(-120:-60:5) arc(0:60:5);
		\end{scope}
		\begin{pgfinterruptboundingbox}
		\path[clip,draw] (6,0) arc(120:180:5) arc(-120:-60:5)arc (0:60:5);
		\end{pgfinterruptboundingbox}
		\foreach \x/\y in {6/0, 6/1,7/-0.5,5/-0.5,6/2.05,7.85/-1.1,4.15/-1.1}
		{\draw (\x,\y) arc(120:180:5) arc(-120:-60:5) arc(0:60:5);}
		\end{scope}
		
		\end{tikzpicture}
	\end{center}
	\caption{Two different sized Mines game boards}
\end{figure}
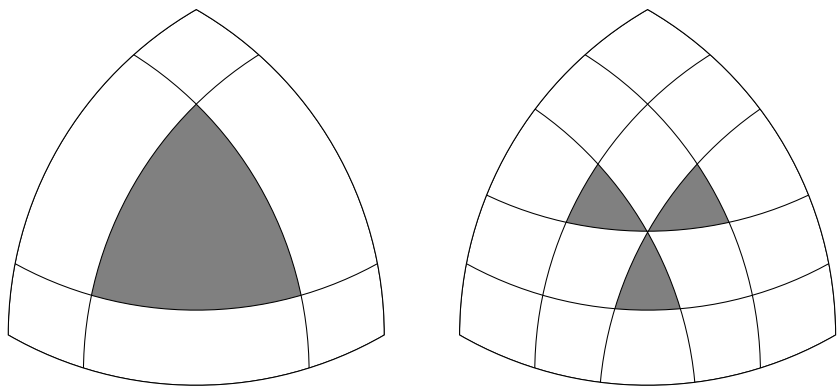

In Section \ref{Section2}, we introduce the games of Tri and Sim. Then we describe how the games can be used to define Ramsey numbers. Near the end of the section we provide a short survey of some known Ramsey numbers and some bounds on unknown Ramsey numbers. The section concludes with a lemma about Tri needed later in the paper. 

In Section \ref{Section3}, in order to help precisely describe the gameboards, we introduce the concept of a triangular set. The remainder of Section \ref{Section3} is devoted to proving combinatorial results related to counting triangular sets. 

In Section \ref{Section4}, we introduce the game of Mines and prove some theorems about its game play. For example, certain variations of Mines have the property that they can never end in a draw.  

In Section \ref{Section5}, we use the game to define the notion of a triangular Ramsey number. Our main results in this section involve finding exact values of some triangular Ramsey numbers and bounds for other triangular Ramsey numbers. Section \ref{Section5} ends by applying the combinatorial results from Section \ref{Section3} and the probabilistic method as pioneered by Erd\H{o}s to find lower bounds for triangular Ramsey numbers. 

In Section \ref{Section6}, we collect together the main results of the paper in Table \ref{TRN}. We then discuss the connection between the triangular Ramsey numbers and the topological Ramsey space $\mathcal{R}_{1}$ introduced by Dobrinen and Todorcevic in \cite{DandT}. We conclude with some open questions and problems related to Mines and triangular Ramsey numbers.

\section{Tri, Sim and Ramsey numbers}\label{Section2}

 Two games that have attracted attention in the literature are  Sim and Tri. The game Sim$_{m}$ was introduced by Simmons in 1969 in \cite{Simmons} and Tri$_{m}$ was introduced by Haggard and Schonberger in 1977 in \cite{haggard1977game}. Both Sim$_{m}$ and Tri$_{m}$ are two player games played on a game board of $m>2$ vertices with $m \choose 2$ possible edges. Each player chooses a color, players alternate turns coloring uncolored edges using their color. Both games end when a monochromatic triangle is constructed (three vertices all of whose edges have the same color) or all edges have been colored. In Tri$_{m}$ the winner is the player that constructs a monochromatic triangle. In Sim$_{m}$ a player wins if they can force the other player to construct a monochromatic triangle. In either game, if no monochromatic triangle is constructed then we say the game ends in a draw. The finite Ramsey theorem for pairs implies that there exists a natural number $m$ such that neither Tri$_{m}$ nor Sim$_{m}$ ever ends in a draw. 
\begin{problem}
 Find the smallest natural number $m>2$ such that neither Tri$_{m}$ nor Sim$_{m}$ ever ends in a draw.
\end{problem}

The solution to Problem 1 is $m=6$.  For $n< m$, the two games have natural generalizations to Tri$_{m}(n)$ and Sim$_{m}(n)$. The only difference being that these versions end when a monochromatic complete graph with $n$ vertices is constructed. In this notation Tri$_{m}$ and Sim$_{m}$ correspond to Tri$_{m}(3)$ and Sim$_{m}(3)$. The Finite Ramsey Theorem for pairs implies that for all natural numbers $n$ there exists a natural number $m$ such that neither Tri$_{m}(n)$ nor Sim$_{m}(n)$ ever ends in a draw. 
\begin{problem} Let $n$ be a natural number greater than 2. Find the smallest natural number $m\ge n$ such that neither Tri$_{m}(n)$ nor Sim$_{m}(n)$ ever ends in a draw.
\end{problem}

The solution to Problem 2 for the natural number $m$ is called the \emph{Ramsey number for $n$} and denoted by $R(n)$. It is known that $R(3)=6$ and $R(4)=18$. However, $R(5)$ still remains unknown.  Figure \ref{Ramsey Numbers} gives upper and lower bounds for some small Ramsey numbers. For example, from the table we have $43\le R(5) \le 49$. In other words, there is a game of Tri$_{43}(5)$ that ends in a draw and no game of Tri$_{49}(5)$ can end in a draw. The lower bound in the second to last row of the table is the lower bound obtain by Erd\H{o}s using the probabilistic method in \cite{ProbMethod}. The last row gives the best known upper and lower bounds.

\begin{figure}[h]
\label{Ramsey Numbers}
\begin{center}
$$\arraycolsep=6pt\def\arraystretch{1.25}
\begin{array}{|c|c|c|c|c|c|c|} \hline
n&$Lower Bound$&R(n)&$Upper Bound$& $References$ \\ \hline 
3&$-$&$6$& $-$&\cite{GG} \\ \hline
$4$&$-$&$18$&$-$&\cite{GG} \\ \hline
$5$&$43$&$?$&$49$ &\cite{Ex} \ \cite{MR} \\ \hline
$6$&$102$&$?$&$165$ &\cite{Ka} \ \cite{Mac} \\ \hline
$7$&$205$&$?$&$540$ &\cite{She} \ \cite{Mac} \\ \hline
n& 2^{n/2} & ? & 4^{n-1}& \cite{ProbMethod} \cite{ProbMethod2} \\ \hline
n& n2^{n/2}[\sqrt{2}/e+o(1)] & ? & n^{-C\frac{\log n}{\log \log n}} 4^{n} & \cite{Spencer} \ \cite{Conlon} \\\hline
\end{array}$$
\end{center}
\caption{Table of upper and lower bounds of some Ramsey numbers and their references.}
\end{figure}

The next Lemma about Tri will be used later to obtain an upper bound for a small triangular Ramsey number. For natural numbers $k$ the notation $R^{k}(n)$ denotes the Ramsey number $\underbrace{R(R(\cdots (R}_{k-times}(n))\cdots))$.
\begin{lemma}\label{TriLemma}
Let $k$ be a natural number. If $k$ games of Tri$_{R^{k}(3)}(3)$ are played on the same game board then there exists a complete graph with three vertices that is monochromatic for each of the $k$ games.
\end{lemma}
\begin{proof}
By the definition of Ramsey number there is complete graph with $R^{k-1}(3)$ vertices that is monochromatic for the first game. Restrict the second game to this complete subgraph of the game board. Again by the definition of Ramsey number there is a complete subgraph of this graph with $R^{k-2}(3)$ vertices that is monochromatic for the first and second games. Continuing this way for $k$ steps we obtain a complete graph with three vertices that is monochromatic for all $n$ games.
\end{proof}

\section{Combinatorics of triangular sets} \label{Section3}
 A \emph{triangular number} is a number that can be represented by a triangular arrangement of equally spaced points. For example, the number 15 can be arranged into a triangle with five levels (see Figure \ref{T5}).
\begin{figure}[h]
\label{T5}
	\begin{center}
		\begin{tikzpicture}
		\foreach \x/\y in {12/0, 13/0, 14/0, 15/0, 16/0, 12.5/0.75, 13.5/0.75, 14.5/0.75, 15.5/0.75, 13/1.5, 14/1.5, 15/1.5, 13.5/2.25, 14.5/2.25, 14/3}
		{
			\node[circle, fill=black, inner sep=0pt, minimum size=5pt] at (\x,\y) {};
		}
		\end{tikzpicture}
	\end{center}
	\caption{Triangular arrangement of 15 points.}
\end{figure}
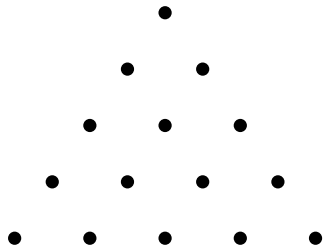\\

For this reason, 15 is called a triangular number. The first four triangular numbers are 1, 3, 6, and 10 whose arrangements are given in Figure \ref{first 4 triangular}. If $T_{n}$ denotes the $n^{\text{th}}$ triangular number then by construction $T_{n+1} = T_{n}+n+1$ with $T_{1} = 1$. It is well known that these numbers can be represented as follows,  $$T_n = \sum\limits_{i=1}^n i = \frac{n(n+1)}{2}={n+1 \choose 2}.$$
\begin{figure}[h]\label{Triangular numbers}
	\begin{center}
		\begin{tikzpicture}
		\foreach \x/\y in {0/0, 1.5/0, 2.5/0, 2/0.75, 4/0, 5/0, 6/0, 4.5/0.75, 5.5/0.75, 5/1.5, 7.5/0, 8.5/0, 9.5/0, 10.5/0, 8/0.75, 9/0.75, 10/0.75, 8.5/1.5, 9.5/1.5, 9/2.25}
		{
			\node[circle, fill=black, inner sep=0pt, minimum size=5pt] at (\x,\y) {};
		}
		\end{tikzpicture}
	\end{center}
	\caption{Triangular arrangement of 1, 3, 6, and 10 points}
	\label{first 4 triangular}
\end{figure}
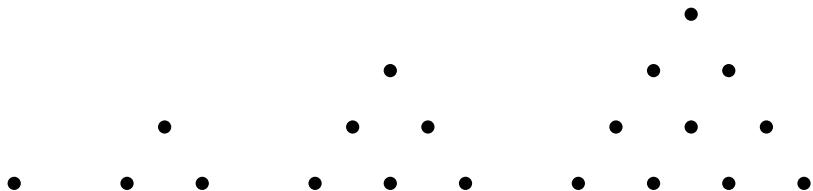\\

	Let $\mathbb{N}=\{1,2,3,\dots\}$ denote the set of natural numbers. A subset $X$ of $\mathbb{N}$ is \emph{triangular} if $|X|$ is a triangular number. Let $\{x_1, x_2, x_3,\ldots, x_{T_n}\}$ be an increasing enumeration of $X$. The numbers in $X$ can be naturally arranged into a triangle with $n$ levels as shown in Figure \ref{levels}.
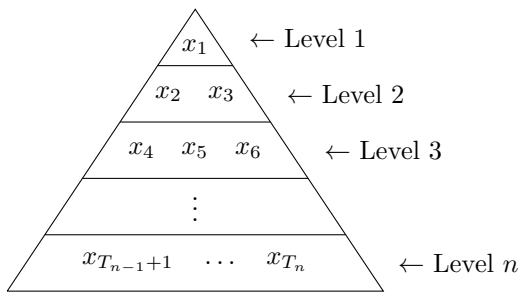
\begin{figure}[h] 
	\begin{center}
		\begin{tikzpicture}
		\draw (2.5,3.25) node{$x_1$};
		\draw (2.5,2.625) node{$x_2\quad x_3$};
		\draw (2.5,1.875) node{$x_4\quad x_5\quad x_6$};
		\draw (2.5,1.2) node{$\vdots$};
		\draw (2.5,0.375) node{$x_{T_{n-1}+1}\quad \ldots\quad x_{T_n}$};
		\draw (0,0) -- (2.5,3.75) -- (5,0) -- (0,0);
		\draw (2,3) -- (3, 3);
		\draw (1.5,2.25) -- (3.5,2.25);
		\draw (1,1.5) -- (4,1.5);
		\draw (0.5,0.75) -- (4.5,0.75);
		\draw (4,3.375) node{$\leftarrow$ Level 1};
		\draw (4.5,2.625) node{$\leftarrow$ Level 2};
		\draw (5,1.875) node{$\leftarrow$ Level 3};
		\draw (6,0.375) node{$\leftarrow$ Level $n$};
		\end{tikzpicture}
	
	\end{center}
	\caption{Triangular arrangement of $\{x_1, x_2, x_3,\ldots, x_{T_n}\}$}
		\label{levels}
\end{figure}
	
	Let $\triangle$ denote the collection of all triangular subsets of $\mathbb{N}$. For $k\in \mathbb{N}$, let $\triangle_k$ denote the triangular sets with $k$ levels \emph{i.e.} those subsets of $\mathbb{N}$ such that $|X|=T_k$. The next partial order is an adaptation of the order on $\mathcal{R}_{1}$ considered by 
Dobrinen and Todorcevic in \cite{DandT} to our current setting. It can be seen as a restriction (to triangular sets) of the partial order used by Laflamme, which inspired the work in \cite{DandT}, to study complete combinatorics in \cite{CompleteCombinatorics}.
\begin{definition}
	For $X,Y\in\triangle$, $X \leqslant Y$ means that $X \subseteq Y$ and every level of $X$ is contained in a single distinct level of $Y$.
\end{definition}

	For example, if we let $W=\{1,2,3,4,5,6\}$, $X=\{2,4,5\}$, $Y=\{3,5,6\}$ and $Z=\{5\}$ then $W\in\triangle_3$, $X,Y\in\triangle_2$ and $Z\in\triangle_1$ such that $Z\subseteq X,Y$ and $X,Y \subseteq W$. Each level of $X$ and $Y$ are contained in a distinct single level of $W$ and $Z$, being only one element, is contained in $X$ and $Y$. Figure \ref{diamond example config} displays this configuration and the associated Hasse diagram in the partial order $(\triangle,\le)$.
	\begin{figure}[h]
		\begin{center}
			\begin{tikzpicture}
			\draw (1.5,1.75) node{$1$};
			\draw (1,1) node{$2$} (2,1) node{$3$};
			\draw (0.5,0.25) node{$4$} (1.5,0.25) node{$5$} (2.5,0.25) node{$6$};
			\draw (-0.2,-0.1) -- (1.5,2.4) -- (3.2,-0.1) -- cycle;
			\draw (0,0) -- (1,1.5) -- (2,0) -- cycle (3,0) -- (1,0) -- (2,1.5) -- cycle (1.1,0.05) -- (1.5,0.675) -- (1.9,0.05) -- cycle;
			\draw (1.5,2.75) node{$W$} (0,0.75) node{$X$} (3,0.75) node{$Y$} (1.5,-0.5) node{$Z$};
			\end{tikzpicture} \hspace{1cm}
				\begin{tikzpicture}[scale=1.75]
			\node (top) at (0,0) {$Z$};
			\node (left) at (-1,1)  {$X$};
			\node (right) at (1,1) {$Y$};
			\node (bottom) at (0,2) {$W$};
			\draw (top) -- (left) -- (bottom) -- (right) -- (top);	
		\end{tikzpicture}

		\end{center}
		\caption{Hasse diagram and possible configuration for $W,X,Y$ and $Z$\label{diamond example config}}
	\end{figure}
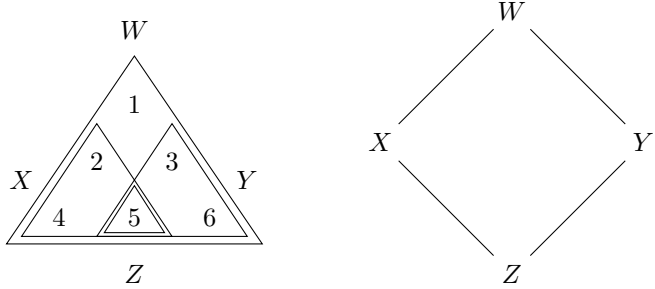

\begin{definition}
	For $k \in \mathbb{N}$ and $X \in \triangle$, let $\triangle_k(X) =\{Y \in \triangle_k : Y \leq X\}$.
\end{definition}

  Suppose $X \in \triangle_3$ and let $\{x_1, x_2, x_3,x_4, x_5, x_6\}$ be an increasing enumeration of $X$. Then $\triangle_2(X)$ contains the elements $\{x_1,x_2,x_3\}$, $\{x_1,x_4,x_5\}$, $\{x_1,x_4,x_6\}$, $\{x_1,x_5,x_6\}$, $\{x_2,x_4,x_5\}$, $\{x_2,x_4,x_6\}$, $\{x_2,x_5,x_6\}$, $\{x_3,x_4,x_5\}$, $\{x_3,x_4,x_6\}$, and $\{x_3,x_5,x_6\}$. Thus for any $X\in \triangle_{3}$, $|\triangle_2(X)|=10$. 

Note that if $Y\in \triangle_4$ then $|\triangle_3(Y)|=41$. To see this, first note that there are ${4\choose 3}=4$ ways to choose three elements from the last row of $Y$. Each one of these possibilities can be added onto any element of $\triangle_2(Y')$ where $Y'$ is the element of $\triangle_3$ obtained by removing the last level of $Y$ to obtain a distinct element of $\triangle_3(Y)$. In particular there are ${4\choose 3}\cdot|\triangle_2(X)|=40$ elements of $\triangle_{3}(Y)$ obtained this way. The only other element of $\triangle_3(Y)$ is $Y'$. So $|\triangle_3(Y)|={4\choose 3}\cdot|\triangle_2(Y')|+1=41$. By a similar argument, one can show that for all $k\in \mathbb{N}$ and for all $Y \in \triangle_k$, if $Y'\in\triangle_{k-1}$ then $$|\triangle_{k-1}(Y)|=1+k|\triangle_{k-2}(Y')|.$$ 

To better express these types of combinatorial relationships we introduce a variant of the binomial coefficient ${n \choose m}$. The next definition should be contrasted with the recursive definition of the binomial coefficients using Pascal's triangle.
\begin{definition}
$$\begin{cases}
&\displaystyle {n \brack k} =\displaystyle {n-1 \brack k} + {n \choose k}{n-1 \brack k-1}\\ \\
&\displaystyle {n \brack 0} =1,\  {n \brack n} = 1
\end{cases}$$
\end {definition}
With this definition note that for all natural numbers $k$, $$ {k \brack k-1} = {k-1 \brack k-1}+{k \choose k-1}\displaystyle {k-1 \brack k-2} =1 + k\displaystyle {k-1 \brack k-2}.$$ Since $\displaystyle {2 \brack 1} =1,$ the argument in the previous paragraph implies that for all natural numbers $k$ and for all $Y \in \triangle_k$, $|\triangle_{k-1}(Y)| = { k \brack k-1}$ as they both satisfy the same recursive formula. The next Theorem generalizes this result.

\begin{theorem}\label{counting}
If $k<n$ and $X\in \triangle_{n}$ then $\displaystyle {n \brack k} =|\triangle_{k}(X)|$. That is, $\displaystyle {n \brack k}$ counts the number of $\triangle_k$'s in a given $\triangle_n$.
\end{theorem}
	\begin{proof} Let $k<n$ and $X\in \triangle_{n}$. We show that $|\triangle_{k}(X)|$ satisfies the same recursive definition as ${n \brack k}$. It is clear that $|\triangle_{n}(X)| =1$. If we consider $\emptyset$ to be the only element of $\triangle_{0}$ then $|\triangle_{0}(X)|= 1$. Thus the base cases of the recursions are the same. We complete the proof by verifying that $|\triangle_{k}(X)| = |\triangle_{k}(X')| + {n \choose k}|\triangle_{k-1}(X')|$ where $X'$ is the triangular set in $\triangle_{n-1}(X)$ obtained by removing the last level of $X$.
	
	Note that	$|\triangle_{k}(X)|$ can be thought of as the number of $\triangle_k$'s in a $\triangle_n$. It should be clear that $|\triangle_{k}(X)| = |\triangle_{k}(X')| + c$ where $c$ is the amount of new $\triangle_k's$ formed when the last level of $X$ is added back to $X'$. Each $\triangle_{k}$ contributing to $c$ must have its last level in the last level of $X$. There are ${n \choose k}$ possibilities for those $k$ points in the final level of $X$. For each collection of $k$ points in the last level of $X$ there are $|\triangle_{k-1}(X')|$ possibilities for triangular sets in $\triangle_{k}(X)$ whose last level is the given $k$ points. Therefore $c = {n \choose k}|\triangle_{k-1}(X')|$ and $|\triangle_{k}(X)| = |\triangle_{k}(X')| + {n \choose k}|\triangle_{k-1}(X')|$.
	\end{proof}
	 Later in the paper we use the next corollary to obtain estimates for upper and lower bounds on ${n \brack k}$. These estimates are needed to apply the probabilistic method to our combinatorial game and obtain lower bounds on triangular Ramsey numbers.
\begin{corollary} \label{brack sums} For $0<k<n$, $\displaystyle {n \brack k} = \sum\limits_{0<i_1<\cdots<i_k\leq{n}} \left( \prod\limits_{j=1}^k {i_j \choose j} \right).$
	\begin{proof} We show that $\sum\limits_{0<i_1<\cdots<i_k\leq{n}}\left( \prod\limits_{j=1}^k {i_j \choose j} \right)$ satisfies the same recursive definition as ${n \brack k}$. By the previous theorem, $\sum\limits_{0<i_1\leq{n}} {i_1 \choose 1} =T_{n}={n \brack 1}$. Clearly, $\sum\limits_{0<i_1<\cdots<i_n\leq{n}} \left( \prod\limits_{j=1}^n {i_j \choose j} \right) = {1 \choose 1}{2 \choose 2}\cdots{n \choose n} = 1={n \brack n}$. Thus the base case of the two recursions are the same.
	Note that
	\begin{align*}&\sum\limits_{0<i_1<\cdots<i_k\leq{n-1}} \left( \prod\limits_{j=1}^k {i_j \choose j} \right) + {n \choose k}
		\cdot\sum\limits_{0<i_1<\cdots<i_{k-1}\leq{n-1}} \left(  \prod\limits_{j=1}^{k-1} {i_j \choose j} \right)\\
		 &=\sum\limits_{0<i_1<\cdots<i_k\leq{n-1}} \left( \prod\limits_{j=1}^k {i_j \choose j} \right) +  \sum\limits_{0<i_1<\cdots<i_{k-1}\leq{n-1}, i_k=n} \left( \prod\limits_{j=1}^{k} {i_j \choose j} \right). \end{align*} The right hand side of the previous equation is just the sum $\sum\limits_{0<i_1<...<i_k\leq{n}} \left( \prod\limits_{j=1}^k {i_j \choose j} \right)$ broken into the parts where $i_k = n$ in the second sum and where it isn't in the first sum, so the equality holds. In particular, the formula satisfies the same recursion formula as ${ n \brack k}$.
	\end{proof}
\end{corollary}

\section{The game of Mines} \label{Section4} The most basic variant of the game of Mines, denoted by Mines$_{3}$, is played on a game board of size $\triangle_3$ with two players each assigned one of two markings and/or colors. Two example game boards are given in Figure \ref{minesgameboard} where the numbers represent the positions that can be marked.  The players have alternating turns in which they may choose to mark one position on the game board or none at all. The game ends when all positions have been played.

\begin{definition}
	Let $X$ denote the set of moves made by player one and $Y$ the set of moves made by player two. Let $x_m$ denote a move done by player one on position $m$ and $y_m$ a move done by player two on position $m$.
\end{definition}
At any point in the game we have $X\cap Y=\emptyset$. If both players always choose to mark a position then given that player one goes first we also have $|X|=|Y|-1$ or $|X|=|Y|$, and $|X|+|Y|=T_3$ when all positions have been played.

\begin{figure}[h]
	\begin{center}
		\begin{tikzpicture}[scale=1.5]
		\draw (1.5,2) node{$1$};
		\draw (1,1.25) node{$2$} (2,1.25) node{$3$};
		\draw (0.5,0.5) node{$4$} (1.5,0.5) node{$5$} (2.5,0.5) node{$6$};
		\draw (0,0.25) -- (1.5,2.5) -- (3,0.25) -- (0,0.25); 
		\draw[fill=gray] (1,0.25) -- (1.5,1) -- (0.5,1) -- cycle (1.5,1) -- (2,0.25) -- (2.5, 1) -- cycle (1.5,1) -- (2,1.75) -- (1,1.75) -- cycle;
		\end{tikzpicture}
		\hspace{1.25cm}
		\begin{tikzpicture}[scale=0.75]
		\path[use as bounding box] (-2.5,-5) rectangle (2, 0);
		\begin{scope}	
		\draw (0,-0.6) node{1} (-1.3,-2.1) node{2} (1.3,-2.1) node{3} (-2,-4.1) node{4} (0,-4.5) node{5} (2,-4.1) node{6};
		\begin{scope}
		\foreach \x/\y in {1/-0.5, -1/-0.5, 0/1}
		{\clip (\x,\y) arc(120:180:5) arc(-120:-60:5) arc(0:60:5);}
		\fill[gray] (0,0) arc(120:180:5) arc(-120:-60:5) arc(0:60:5);
		\end{scope}
		\begin{pgfinterruptboundingbox}
		\path[clip,draw] (0,0) arc(120:180:5) arc(-120:-60:5) arc (0:60:5);
		\end{pgfinterruptboundingbox}
		\foreach \x/\y in {0/0, 0/1,1/-0.5,-1/-0.5}		
		{\draw (\x,\y) arc(120:180:5) arc(-120:-60:5) arc(0:60:5);}	
		\end{scope}
		\end{tikzpicture}
	\caption{Two alternative gameboards for Mines$_3$.}
	\label{minesgameboard}
	\end{center}
\end{figure}
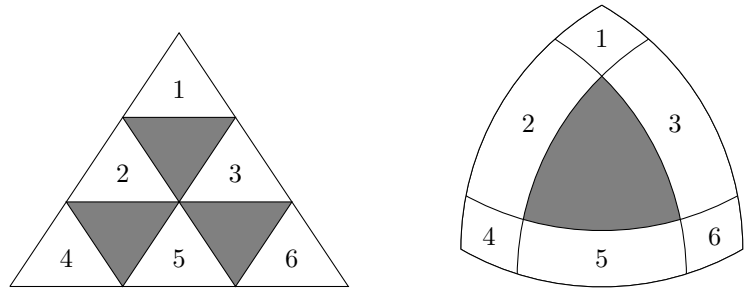

 Player one \emph{wins} if $\triangle_{2}(X)\not=\emptyset$. Player two \emph{wins} if $\triangle_{2}(Y)\not=\emptyset$.  A \emph{draw} occurs if all positions have been played and neither player has won the game. That is, $\triangle_2(X)=\triangle_2(Y)=\emptyset$ and $|X|+|Y|=T_3$. 

Figure \ref{ExampleGame} describes a possible game such that $X=\{x_2,x_3,x_4\}$ and $Y=\{y_1,y_5,y_6\}$. Here $\triangle_2(Y)\not=\emptyset$ since $\{y_1,y_5,y_6\}\in\triangle_{2}$ indicating that player two has won the game.
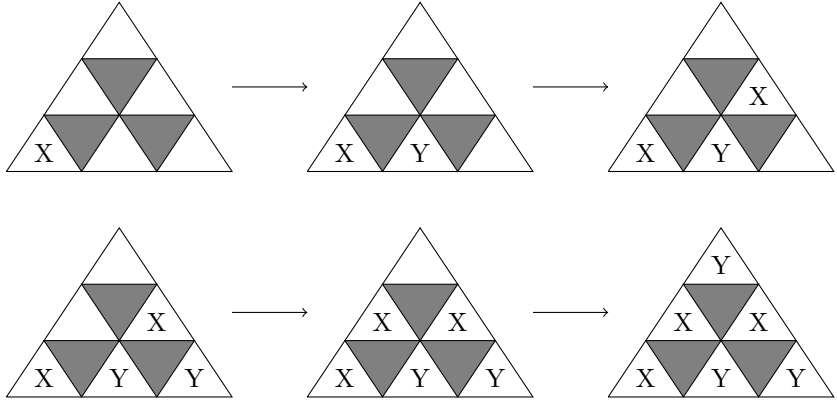
\begin{figure}[h] 
	\begin{center}
		\begin{tikzpicture} 
		\draw (0.5,0.25) node{X};
		\draw (0,0) -- (1.5,2.25) -- (3,0) -- (0,0);
		\draw[fill=gray] (1,0) -- (1.5,0.75) -- (0.5,0.75) -- cycle (1.5,0.75) -- (2,0) -- (2.5, 0.75) -- cycle (1.5,0.75) -- (2,1.5) -- (1,1.5) -- cycle;
		
		\draw (4.5,0.25) node{X} (5.5,0.25) node{Y};
		\draw (4,0) -- (5.5,2.25) -- (7,0) -- (4,0);			
		\draw[fill=gray] (5,0) -- (5.5,0.75) -- (4.5,0.75) -- cycle (5.5,0.75) -- (6,0) -- (6.5, 0.75) -- cycle (5.5,0.75) -- (6,1.5) -- (5,1.5) -- cycle;
		
		\draw (10,1) node{X} (8.5,0.25) node{X} (9.5,0.25) node{Y};
		\draw (8,0) -- (9.5,2.25) -- (11,0) -- (8,0);			
		\draw[fill=gray] (9,0) -- (9.5,0.75) -- (8.5,0.75) -- cycle (9.5,0.75) -- (10,0) -- (10.5, 0.75) -- cycle (9.5,0.75) -- (10,1.5) -- (9,1.5) -- cycle;	
			
		\draw (2,-2) node{X} (0.5,-2.75) node{X} (1.5,-2.75) node{Y} (2.5,-2.75) node{Y};
		\draw (0,-3) -- (1.5,-0.75) -- (3,-3) -- (0,-3);
		\draw[fill=gray] (1,-3) -- (1.5,-2.25) -- (0.5,-2.25) -- cycle (1.5,-2.25) -- (2,-3) -- (2.5, -2.25) -- cycle (1.5,-2.25) -- (2,-1.5) -- (1,-1.5) -- cycle;
		
		\draw (5,-2) node{X} (6,-2) node{X} (4.5,-2.75) node{X} (5.5,-2.75) node{Y} (6.5,-2.75) node{Y};
		\draw (4,-3) -- (5.5,-0.75) -- (7,-3) -- (4,-3);
		\draw[fill=gray] (5,-3) -- (5.5,-2.25) -- (4.5,-2.25) -- cycle (5.5,-2.25) -- (6,-3) -- (6.5, -2.25) -- cycle (5.5,-2.25) -- (6,-1.5) -- (5,-1.5) -- cycle;
		
		\draw (9.5,-1.25) node{Y} (9,-2) node{X} (10,-2) node{X} (8.5,-2.75) node{X} (9.5,-2.75) node{Y} (10.5,-2.75) node{Y};
		\draw (8,-3) -- (9.5,-0.75) -- (11,-3) -- (8,-3);
		\draw[fill=gray] (9,-3) -- (9.5,-2.25) -- (8.5,-2.25) -- cycle (9.5,-2.25) -- (10,-3) -- (10.5, -2.25) -- cycle (9.5,-2.25) -- (10,-1.5) -- (9,-1.5) -- cycle;
		
		\draw [->] (3,1.125) -- (4,1.125);
		\draw [->] (7,1.125) -- (8,1.125);
		\draw [->] (3,-1.875) -- (4,-1.875);
		\draw [->] (7,-1.875) -- (8,-1.875);
		\end{tikzpicture}
	\end{center}
	\caption{An example game of Mines$_3$.}
	\label{ExampleGame}
\end{figure}
\begin{theorem}
	Both players cannot construct a $\triangle_2$. That is, it is impossible for both $\triangle_2(X)\not=\emptyset$ and $\triangle_2(Y)\not=\emptyset$.\label{Mines3 no draw 1}
	\end{theorem}
	\begin{proof}
		Assume this is not the case. In other words, there exists a situation such that $\triangle_2(X)\not=\emptyset$ and $\triangle_2(Y)\not=\emptyset$. Without loss of generality, we may assume that all positions have been played.
		
		The largest row of the game board has 3 positions and the largest level of a $\triangle_2$ has 2 positions. Therefore, both $X$ and $Y$ cannot construct their $\triangle_2$'s largest level in the same row of the game board. By the pigeon hole principle there exists $a,b \in \{4,5,6\}$ such that $x_a,x_b\in X$ or $y_a,y_b\in Y$. If $x_a,x_b\in X$ then $y_2,y_3\in Y$. In order for $\triangle_{2}(X)\not=\emptyset$ we must have $x_1\in X$ and we find $\triangle_{2}(Y)=\emptyset$, a contradiction. If instead we had $y_a,y_b\in Y$ then $x_2,x_3\in X$. Once again, in order for $\triangle_{2}(X)\not=\emptyset$ we must have $x_1\in X$ and we find $\triangle_3(Y)=\emptyset$, a contradiction. Therefore it is impossible for both players to construct a $\triangle_2$.
	\end{proof}

\begin{theorem}
Mines$_{3}$ never ends in a draw.
\label{Mines3 no draw 2}
\end{theorem}
\begin{proof}
Assume a full game of Mines$_{3}$ has been played. By the pigeonhole principle, at least two of the bottom three elements must be of the same color. If any of the three elements above the bottom row are of the same color as the two on the bottom, then a $\triangle_{2}$ has been constructed in that color. For this not to happen, the three elements in the top two rows must all be in the opposite color. If this is the case, then a $\triangle_{2}$ has been constructed in the opposite color. Therefore, it is impossible for a game of Mines$_{3}$ to be played in which neither player constructs a $\triangle_{2}$, and no game can be played in which both players construct a $\triangle_{2}$.
\end{proof}

An interesting variation of the game play exploits the fact that the game board has $120\degree$ rotational symmetry about its center. There are three directions to it given by the perpendicular from any of the three edges to its adjacent vertex. We let $D_n$ denote the orientation of the game board in the $n$ direction (see Figure \ref{directions}). We use the notation $\triangle_{k,n}(X)$ to denote all $\triangle_k$'s on the game board in the $D_{n}$ direction. For example, $\{1,3,4\}\in \triangle_{2,2}(X)$ but $\{1,3,4\}\not\in \triangle_{2,1}(X)$. Note that if $\triangle_{k,n}(X)\not=\emptyset$ then  there exists a $\triangle_k$ in the direction of $D_n$ contained in $X$. Likewise, if $\triangle_{k,n}(Y)\not=\emptyset$ then  there exists a $\triangle_k$ in the direction of $D_n$ contained in $Y$. The player that wins in two of the three directions wins this variation of the game.

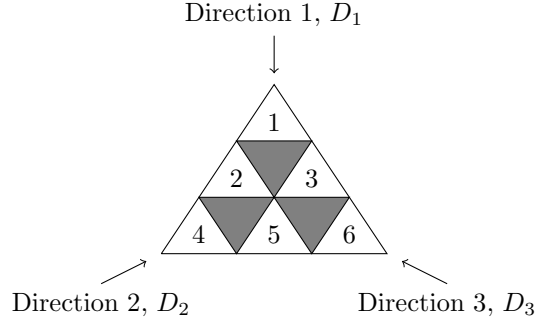
\begin{figure}[h] 
	\begin{center}
		\begin{tikzpicture}
		\draw (1.5,1.75) node{$1$};
		\draw (1,1) node{$2$} (2,1) node{$3$};
		\draw (0.5,0.25) node{$4$} (1.5,0.25) node{$5$} (2.5,0.25) node{$6$};
		\draw (0,0) -- (1.5,2.25) -- (3,0) -- (0,0); 
		\draw[fill=gray] (1,0) -- (1.5,0.75) -- (0.5,0.75) -- cycle (1.5,0.75) -- (2,0) -- (2.5, 0.75) -- cycle (1.5,0.75) -- (2,1.5) -- (1,1.5) -- cycle;
		\draw [->] (1.5,2.9) node[above] {Direction $1$, $D_{1}$} -- (1.5,2.4);
		\draw [->] (-0.8,-0.4) node[below] {Direction $2$, $D_{2}$} -- (-0.2,-.1);
		\draw [->] (3.8,-0.4) node[below] {Direction $3$, $D_{3}$} -- (3.2,-.1);
		\end{tikzpicture}
	\end{center}
	\caption{A game board for Mines$_{3}$ with directions $D_{1}$, $D_{2}$ and $D_{3}$ labeled.}
	\label{directions}
\end{figure}

By Theorem \ref{Mines3 no draw 1} and Theorem \ref{Mines3 no draw 2} we see that there can never be a draw in a single direction. In addition, given that there are three directions one player must win in at least two directions. Therefore this variation of Mines$_{3}$ can never end in a draw. The next Theorem is true for both variations of Mines$_{3}$. We give the proof for the omnidirectional case as it is more interesting.
\begin{theorem}
	Player one has a winning strategy for Mines$_{3}$.
	\end{theorem}
	\begin{proof} In the first three moves of the game we have $|X|=2$ and $|Y|=1$. Therefore player one can guarantee there exists $q,r\in \{1,4,6\}$ such that $x_{q},x_{r}\in X$.
	
		By rotating the game board $120\degree$ to the left or right, we can without loss of generality, assume $x_4,x_6\in X$. At this point in the game $\{y_1, y_2, y_3\}\not\subseteq Y$ and $\{x_1, x_2, x_3\}\not\subseteq X$ since $|X|+|Y|=3$. On the next move player one plays position $1$, $2$ or $3$ whichever is available. Player one wins in direction $D_{1}$ since $\triangle_{3,1}(X)\not=\emptyset$. Thus player one only needs to win in one other direction to win the game. If $x_1\in X$ then $\{x_1,x_4,x_6\}=X$ and $\triangle_{3,2}(X)\not=\emptyset$. If $x_2\in X$ then $\{x_2,x_4,x_6\}=X$ and $\triangle_{3,3}(X)\not=\emptyset$. If $x_3\in X$ then $\{x_3,x_4,x_6\}=X$ and $\triangle_{3,2}( X)\not=\emptyset$. Thus, in any case, player one wins in at least two out of the three directions. 
	\end{proof}

\subsection{The game of Mines$_{m}(p,q,k)$}
Let $p,q,m$ and $k$ be positive integers with $k\le p,q\le m$. The game Mines$_{m}(p,q,k)$ is played on a game board of size $\triangle_{m}$ with two players each assigned one of two markings and/or colors. Players have alternating turns in which they may choose to mark one position, in this case positions are the $\triangle_{k}$'s on the game board, or not mark a position. If $p=q=n$ then we use the notation Mines$_{m}(n,k)$. If $p=q=n$ and $k=1$ then we use the notation Mines$_{m}(n)$. In this notation, Mines$_{3} = $ Mines$_{3}(2) = $ Mines$_{3}(2,1) = $ Mines$_{3}(2,2,1)$ since the positions played by both players in Mines$_{3}$ are the $\triangle_{1}$'s on the game board. 

	We again let $X\subseteq\triangle_{k}$ denote the set of moves made by player one and $Y\subseteq\triangle_{k}$ the set of moves made by player two. At any point in the game we have $X\cap Y=\emptyset$. If both players always choose to mark a position then given that player one goes first we also have $|X|=|Y|-1$ or $|X|=|Y|$, and $|X|+|Y|={m \brack k}$ when all positions have been played.  Player one \emph{wins} if they construct a $Z\in\triangle_{p}$ such that $\triangle_{k}(Z)\subseteq X$ before the second player is able to construct a $Z\in\triangle_{q}$ such that $\triangle_{k}(Z)\subseteq Y$.  Player two \emph{wins} if they construct a $Z\in\triangle_{q}$ such that $\triangle_{k}(Z)\subseteq Y$ before the first player is able to construct a $Z\in\triangle_{p}$ such that $\triangle_{k}(Z)\subseteq X$. A \emph{draw} occurs if all positions have been played and neither player has won the game. That is, $|X|+|Y|={ m \brack k}$ and for all $Z\in\triangle_{n}$, $\triangle_{k}(Z)\not \subseteq X$ and $\triangle_{k}(Z)\not \subseteq Y$. The game ends when all positions have been played or one of the players wins. 
	
	For some small game boards it is unnecessary to keep track of which player first constructs the winning triangular set. Instead the players can simply fill out the game board completely and then check to see who wins. For example, Mines$_{5}(3) = $ Mines$_{5}(3,3,1)$ has this property. Figure \ref{MinesBoard} gives two examples of game boards for Mines$_{5}(n)$.  
\begin{figure}[h]
\begin{center}
			\begin{tikzpicture}
		\path[use as bounding box] (-2.5,-5) rectangle (8.5, 0);
		\begin{scope}[yshift=-4.5cm, xshift=-3cm]
		\draw (2.5,3.25) node{$1$};
		\draw (2,2.5) node{$2$} (3,2.5) node{$3$};
		\draw (1.5,1.75) node{$4$} (2.5,1.75) node{$5$} (3.5,1.75) node{$6$};
		\draw (1,1) node{$7$} (2,1) node{$8$} (3,1) node{$9$} (4,1) node{$10$};
		\draw (0.5,0.25) node{$11$} (1.5,0.25) node{$12$} (2.5,0.25) node{$13$} (3.5,0.25) node{$14$} (4.5,0.25) node{$15$};
		\draw (0,0) -- (2.5,3.75) -- (5,0) -- (0,0); 
		\draw[fill=gray] (1,0) -- (1.5,0.75) -- (0.5,0.75) -- cycle (1.5,0.75) -- (2,0) -- (2.5, 0.75) -- cycle (1.5,0.75) -- (2,1.5) -- (1,1.5) -- cycle (3,0) -- (2.5,0.75) -- (3.5,0.75) -- cycle (4,0) -- (3.5,0.75) -- (4.5,0.75) -- cycle (2.5,0.75) -- (2, 1.5) -- (3,1.5) -- cycle (3.5,0.75) -- (3,1.5) -- (4,1.5) -- cycle (2,1.5) -- (1.5,2.25) -- (2.5,2.25) -- cycle (3,1.5) -- (2.5,2.25) -- (3.5,2.25) -- cycle (2.5,2.25) -- (2,3) -- (3,3) -- cycle;
	\end{scope}
		\begin{scope}	
		\draw (6,-0.6) node{$1$} (5.25,-1.25) node{$2$} (6.75,-1.25) node{$3$} (4.6,-2.1) node{$4$} (6,-2.1) node{$5$} (7.4,-2.1) node{$6$} (4.2,-3.2) node{$7$} (5.2,-3.4) node{$8$} (6.8,-3.4) node{$9$} (7.8,-3.2) node{$10$} (4,-4.1) node{$11$} (5,-4.4) node{$12$} (6,-4.5) node{$13$} (7,-4.4) node{$14$} (8,-4.1) node{$15$};
		\begin{scope}
		\foreach \x/\y in {5/-0.5, 6/2.05, 7.85/-1.1}
		{\clip (\x,\y) arc(120:180:5) arc(-120:-60:5) arc(0:60:5);}
		\fill[gray] (6,0) arc(120:180:5) arc(-120:-60:5) arc(0:60:5);
		\end{scope}
		\begin{scope}
		\foreach \x/\y in {6/1, 4.15/-1.1, 7.85/-1.1}
		{\clip (\x,\y) arc(120:180:5) arc(-120:-60:5) arc(0:60:5);}
		\fill[gray] (6,0) arc(120:180:5) arc(-120:-60:5) arc(0:60:5);
		\end{scope}
		\begin{scope}
		\foreach \x/\y in {7/-0.5, 6/2.05, 4.15/-1.1}
		{\clip (\x,\y) arc(120:180:5) arc(-120:-60:5) arc(0:60:5);}
		\fill[gray] (6,0) arc(120:180:5) arc(-120:-60:5) arc(0:60:5);
		\end{scope}
		\begin{pgfinterruptboundingbox}
		\path[clip,draw] (6,0) arc(120:180:5) arc(-120:-60:5)arc (0:60:5);
		\end{pgfinterruptboundingbox}
		\foreach \x/\y in {6/0, 6/1,7/-0.5,5/-0.5,6/2.05,7.85/-1.1,4.15/-1.1}
		{\draw (\x,\y) arc(120:180:5) arc(-120:-60:5) arc(0:60:5);}
		\end{scope}
		
		\end{tikzpicture}
	\end{center}
	\caption{Two example game boards of size $\triangle_5$}
	\label{MinesBoard}
\end{figure}

\begin{theorem}
	In the game of Mines$_{5}(3)$,	both players cannot construct a $\triangle_3$. That is, it is impossible for both $\triangle_3(X)\not=\emptyset$ and $\triangle_3(Y)\not=\emptyset$.
	\end{theorem}
	\begin{proof}
		Assume this is not the case. There exists a situation such that $\triangle_3(X)\not=\emptyset$ and $\triangle_3(Y)\not=\emptyset$. The largest level of a $\triangle_5$ has 5 positions and the largest level of a $\triangle_3$ has 3 positions. Therefore, both $X$ and $Y$ cannot construct their $\triangle_3$'s largest level in the same row of the game board. By the pigeon hole principle there exists $a,b,c \in \{11,12,13,14,15\}$ such that $x_a,x_b,x_c\in X$ or $y_a,y_b,y_c\in Y$. Suppose that $x_a,x_b,x_c\in X$. We have the following two cases:
		
		\emph{Case 1:} There exists $d,e,f\in \{7,8,9,10\}$ such that $y_d,y_e,y_f\in Y$. Since $\triangle_3(Y)\not=\emptyset$, either there exists $g,h \in \{4,5,6\}$ such that  $y_g,y_h\in Y$ or $y_{2},y_{3}\in Y$. If $g,h \in \{4,5,6\}$ then $x_2,x_3\in X$. If $y_{2},y_{3}\in Y$ then there exists $i,j\in \{4,5,6\}$ such that $x_i,x_j\in X$. In order for $\triangle_3(X)\not=\emptyset$ we must have $x_1\in X$ which causes $\triangle_3(Y)=\emptyset$, a contradiction.
		
		\emph{Case 2:} There exists $d,e,f\in \{4,5,6\}$ such that $y_d,y_e,y_f\in Y$. Then either there exists $g,h \in \{7,8,9,10\}$ such that $x_g,x_h\in X$ or $x_{2},x_{3}\in X$. If $x_{2},x_{3}\in X$ then clearly $\triangle_3(Y)\not=\emptyset$. If there exists $g,h \in \{7,8,9,10\}$ such that $x_g,x_h\in X$ then we must have $y_2,y_3\in Y$. In order for $\triangle_3(X)\not=\emptyset$ we must have $x_1\in X$ which causes $\triangle_3(Y)=\emptyset$, a contradiction.
		
		If instead $y_a,y_b,y_c\in Y$ then a similar argument gives a contradiction. So it is impossible for both $\triangle_3(X)\not=\emptyset$ and $\triangle_3(Y)\not=\emptyset$.
	\end{proof}
	The previous Theorem can be extended, in the game of Mines$_{2n-1}(n)$	both players cannot construct a $\triangle_n$. That is, it is impossible for both $\triangle_n(X)\not=\emptyset$ and $\triangle_n(Y)\not=\emptyset$. We leave the proof to the interested reader.

\section{Triangular Ramsey numbers}\label{Section5}

The Finite Ramsey Theorem for $\mathcal{R}_{1}$ which follows from the work of Dobrinen and Todorcevic in \cite{DandT} can be used to show that for all $p,q,k\in\mathbb{N}$ with $k\le p,q$ there is a game board of size $\triangle_m$ with $m\ge p,q$ such that Mines$_{m}(p,q,k)$ never ends in a draw. 
\begin{problem}
 Let $p,q$ and $k$ be natural numbers such that $k\le p,q$. Find the smallest natural number $m\ge p,q$ such that Mines$_{m}(p,q,k)$ never ends in a draw.
\end{problem} 

The solution to the Problem when $p,q=2$ and $k=1$ is $m=3$. The solution to the Problem for the natural numbers $p,q$ and $k$ is called the \emph{triangular Ramsey number for $p,q$ and $k$} and denoted by $\mathcal{R}_{1}(p,q,k)$. If $p=q=n$ then we denote $\mathcal{R}_{1}(p,q,k)$  by $\mathcal{R}_{1}(n,k)$.
\begin{lemma}
For all natural numbers $p$ and $q$, $\mathcal{R}_{1}(p,q,1)>p+q-2$.
\end{lemma}
\begin{proof}
Suppose a game of Mines$_{p+q-2}(p,q,1)$ is to be played. It is possible for the bottom row, of size $p+q-2$, to contain $p-1$ elements in $X$ and the remaining $q-1$ elements in $Y$. Note that neither player has constructed the bottom row of a winning triangle in the bottom row. Since the rows decrease in size as players move up the triangle, it is possible for both players to fail to construct the bottom row of their winning triangle anywhere on the board. Therefore $\mathcal{R}_{1}(p,q,1)>p+q-2$.
\end{proof}

\begin{theorem}\label{Delta1Thm}
For all natural numbers $p$ and $q$, $\mathcal{R}_{1}(p,q,1)=p+q-1$. In particular, for all numbers $n$, $\mathcal{R}_{1}(n,1) = 2n-1$.
\end{theorem}
\begin{proof} By the previous Lemma $\mathcal{R}_{1}(p,q,1)>p+q-2$. So the result follows by showing via induction on $p+q$ that $\mathcal{R}_{1}(p,q,1)\le p+q-1$. The base case occurs when $p=q=1$, \emph{i.e.} $p+q=2$. The base case is trivial; the first person to color an element creates a complete $\triangle_1$ and wins. In other words, $\mathcal{R}_{1}(1,1,1)=1$.

 Now suppose $p+q=n+1$ and the results holds when $p+q=n$. Note that $(p-1)+q=n$ and $p+(q-1)=n$. Thus by the inductive hypothesis, $\mathcal{R}_{1}(p-1,q)=\mathcal{R}_{1}(p,q-1)\le p+q-2$. We can now prove by contradiction that no game of Mines$_{p+q-1}(p,q,1)$ ends in a draw. Toward a contradiction suppose a full game of Mines$_{p+q-1}(p,q,1)$ has been played on a board of height $n$ and ends in a draw, \emph{i.e.} $\triangle_{p}(X)=\emptyset$ and $\triangle_{q}(Y)=\emptyset$. Since $p+q-1=n$ and the bottom level of the game board contains $n$ positions, the pigeon hole principle mandates that the bottom row of the game board must contain either $p$ elements in $X$, considered \emph{Case 1}, or $q$ elements in $Y$, considered \emph{Case 2}. 

\emph{Case 1:} The bottom row of the game board contains $Z'$, a set of $p$ elements in $X$. By the equation above, $\mathcal{R}_{1}(p-1,q)=p+q-1$, and the inductive hypothesis, we can see that there exists either $Z\in\triangle_{p-1}(X) \not= \emptyset$ or $\zeta\in\triangle_{q}(Y) \not= \emptyset$. If the first is the case, then $Z\cup Z'\in\triangle_{p}(X)$ and the game is won by player 1. If the second is the case, then the game is won by player 2. In either, we have a contradiction.

\emph{Case 2:} The bottom row of the game board contains $q$ elements in $Y$. By a similar method to above, we can show that either $Z\in\triangle_{q-1}(Y) \not= \emptyset$ or $\zeta\in\triangle_{p}(X)$ exists. In either situation, a fully colored triangle is made and we have a contradiction.

In either case, we obtain a contradiction. Therefore $\mathcal{R}_{1}(p,q,1)\le p+q-1$ when $p+q=n+1$. 
\end{proof}
Next we define a sequence $(M_{n,k})$ which we use to establish upper bounds for triangular Ramsey numbers. We let $R(n,k)$ denote the smallest size, number of vertices, of a complete graph such that for any coloring of its complete subgraphs with $k$ vertices with two colors there exists a complete subgraph with $n$ vertices where the coloring is monochromatic. The existence of these Ramsey numbers also follows from Ramsey's Theorem and could also be introduced by generalizing the game of Tri. Here we let $R^{l}(n,k)$ denote $\underbrace{R(R(\cdots (R}_{l-times}(n,k),k)\cdots),k)$.
\begin{equation}
\label{Mnk Def}
\begin{cases}
M_{n,k}=n & \text{if} \ k=1, \\
M_{n,k}=n & \text{if} \ n= k,\\
M_{n+1,k} = \displaystyle R^{{\mathcal{R}_{1}(M_{n,k},k-1) \brack k-1}}(n+1,k) &\text{if} \ n> k>1.\end{cases}
\end{equation}

\begin{theorem}\label{MTheorem}
Let $\left(M_{n,k}\right)$ be the sequence recursively defined by $(\ref{Mnk Def})$. If $n\ge k$ then 
$$ \mathcal{R}_{1}(n,k)\le M_{2n-1,k}.$$
\end{theorem}
\begin{proof}

We begin by establishing a simpler result by induction on $n$.
\begin{claim} Suppose that a game of Mines$_{M_{n,k}}(n,k)$ is played to completion. There exists a $Z\in \triangle_{n}$ on the game board such that given any level $i$ of $Z$, either all $\triangle_{k}$'s contained in $Z$ whose last level is contained in the $i^{th}$ level of $Z$ are played by player 1 or all $\triangle_{k}$'s contained in $Z$ whose last level is contained in the $i^{th}$ level of $Z$ are played by player 2.
\end{claim}
\begin{proof}
First note that by the previous Theorem $\mathcal{R}_{1}(n,1)=2n-1 = M_{2n-1,1}$. Thus, the Claim holds when $k=1$. Next fix $k\ge 2$. Consider the base case when $n=k$ and $M_{n,k}=n$. If a game of Mines$_{n}(n,k)$ has been played to completion then, since there is only one playable position on the game board, the base case holds trivially. 

Assume that the Claim holds for $M_{n,k}$. Suppose that a game of Mines$_{M_{n+1,k}}(n+1,k)$ has been played to completion. As usual let $X$ denote the moves made by player one and $Y$ denote those made by player two. For each element $Z\in \triangle_{k-1}$ in the first $\mathcal{R}_{1}(M_{n,k},k-1)$ levels of the game board, we play a game of Tri$_{M_{n+1,k}}(n+1,k)$ on the final level of the game board as follows: player 1 plays the $k$-element set $\{i_{1},i_{2},\dots i_{k}\}$ if $Z\cup\{i_{1},i_{2},\dots,i_{k}\}\in X$ and player 2 plays the $k$-element set $\{i_{1},i_{2},\dots i_{k}\}$ if $Z\cup\{i_{1},i_{2},\dots,i_{k}\}\in Y$. By argument similar to the proof of  Lemma \ref{TriLemma}, there exists an $(n+1)$-element set in the last level of the game board $\{z_{1},z_{2},\dots z_{n+1}\}$ such that for all $Z\in \triangle_{k-1}$ in the first $\mathcal{R}_{1}(M_{n,k},k-1)$ levels of the game board either $(\dagger)$ for all $i_{1},i_{2},\dots i_{k}\in\{z_{1},z_{2},\dots z_{n+1}\}$, $Z\cup\{i_{1},i_{2},\dots i_{k}\}\in X$ or $(\ddagger)$  for all $i_{1},i_{2},\dots i_{k}\in\{z_{1},z_{2},\dots z_{n+1}\}$, $Z\cup\{i_{1},i_{2},\dots i_{k}\}\in Y$.

Next consider the following hypothetical game of Mines$_{\mathcal{R}_{1}(M_{n,k},k-1)}(M_{n,k},k-1)$ played on the first $\mathcal{R}_{1}(M_{n,k},k-1)$ levels of the our original game board. Let $\bar{X}$ denote the moves made by player one and $\bar{Y}$ denote those made by player two. In this game, player 1 plays position $Z\in \triangle_{k-1}$ if $(\dagger)$ holds and player 2 plays position $Z\in \triangle_{k-1}$ if $(\ddagger)$ holds. By definition this game does not end in a draw. If player 1 wins this game then there exists $W\in \triangle_{M_{n,k}}$ such that all $\triangle_{k}$'s whose first $k-1$ levels are in $W$ and whose last level is contained in $\{z_{1},z_{2},\dots z_{n+1}\}$ are played by player 1. If player 2 wins this game then there exists $W\in \triangle_{M_{n,k}}$ such that all $\triangle_{k}$'s whose first $k-1$ levels are in $W$ and whose last level is contained in $\{z_{1},z_{2},\dots z_{n+1}\}$ are played by player 2.

By the induction hypothesis there exists a $Z'\in \triangle_{n}(W)$ such that given any level $i$ of $Z'$, either all $\triangle_{k}$'s contained in $Z'$ whose last level is contained in the $i^{th}$ level of $Z'$ are played by player 1 or all $\triangle_{k}$'s contained in $Z'$ whose last level is contained in the $i^{th}$ level of $Z'$ are played by player 2.

Let $Z''= Z'\cup\{z_{1},z_{2},\dots z_{n+1}\}\in \triangle_{n+1}$. Then either all $\triangle_{2}$'s contained in $Z''$ whose last level is contained in the $i^{th}$ level of $Z''$ are played by player 1 or all $\triangle_{2}$'s contained in $Z''$ whose last level is contained in the $i^{th}$ level of $Z''$ are played by player 2. Therefore the Claim holds by induction.
\end{proof}

To prove the inequality, assume toward a contradiction that a game of Mines$_{M_{2n-1,k}}$ ends in a draw. By the previous Claim there exists a $Z\in \triangle_{2n-1}$ on the game board such that given any level $i$ of $Z$, either all $\triangle_{k}$'s contained in $Z$ whose last level is contained in the $i^{th}$ level of $Z$ are played by player 1 or all $\triangle_{k}$'s contained in $Z$ whose last level is contained in the $i^{th}$ level of $Z$ are played by player 2. By the pigeon hole principle there are either at least $n$ level where player 1 plays all $\triangle_{k}$'s or at least $n$ levels where player 2 plays all $\triangle_{k}$'s. If there are at least $n$ levels where player 1 wins then any $W\in\triangle_{n}$ whose levels come from these $n$ levels witnesses a win for player 1, a contradiction. Similarly, if there are at least $n$ levels where player 2 wins then any $W\in\triangle_{n}$ whose levels come from these $n$ levels witnesses a win for player 2, a contradiction. Therefore, this game could not have ended in a draw.
\end{proof}
When $n=k+1$ the previous proof can be simplified and we obtain smaller upper bounds. In fact, in this special case, induction on $n$ is unnecessary.

\begin{theorem}\label{MTheorem2} Suppose that $k\ge2$. Then 
$$ \mathcal{R}_{1}(k+1,k)\le R^{{\mathcal{R}_{1}(k+1, k-1) \brack k-1}}(k+1,k).$$
\end{theorem}
\begin{proof} Let $M=R^{\mathcal{R}_{1}(k+1, k-1) \brack k-1}(k+1,k)$. Toward a contradiction suppose that a game of Mines$_{M}(k+1,k)$ ends in a draw. As usual, let $X$ denote the moves made by player one and $Y$ denote those made by player two. For each element $Z\in \triangle_{k-1}$ in the first $\mathcal{R}_{1}(k+1, k-1)$ levels of the game board, we play a game of Tri$_{M}(k+1,k)$ on the final level of the game board as follows: player 1 plays the $k$-element set $\{i_{1},i_{2},\dots i_{k}\}$ if $Z\cup\{i_{1},i_{2},\dots,i_{k}\}\in X$ and player 2 plays the $k$-element set $\{i_{1},i_{2},\dots i_{k}\}$ if $Z\cup\{i_{1},i_{2},\dots,i_{k}\}\in Y$. By Lemma \ref{TriLemma}, there exists an $(k+1)$-element set in the last level of the game board $\{z_{1},z_{2},\dots z_{k+1}\}$ such that for all $Z\in \triangle_{k-1}$ in the first $\mathcal{R}_{1}(k+1, k-1)$ levels of the game board either $(\dagger)$ for all $i_{1},i_{2},\dots i_{k}\in\{z_{1},z_{2},\dots z_{k+1}\}$, $Z\cup\{i_{1},i_{2},\dots i_{k}\}\in X$ or $(\ddagger)$  for all $i_{1},i_{2},\dots i_{k}\in\{z_{1},z_{2},\dots z_{k+1}\}$, $Z\cup\{i_{1},i_{2},\dots i_{k}\}\in Y$.

Next consider the following hypothetical game of Mines$_{\mathcal{R}_{1}(k+1, k-1)}(k+1,k-1)$). Let $\bar{X}$ denote the moves made by player one and $\bar{Y}$ denote those made by player two. In this game, player 1 plays position $Z$ if $(\dagger)$ holds and player 2 plays position $Z$ if $(\ddagger)$ holds. In other words, $\bar{X}=\{Z\in \triangle_{k-1} :Z\cup\{i_{1},i_{2},\dots i_{k}\}\in X\}$ and $\bar{Y}=\{Z\in \triangle_{k-1} :Z\cup\{i_{1},i_{2},\dots i_{k}\}\in Y\}$. By definition this game does not end in a draw. 

Suppose that player 1 wins and let $W$ be some element of $\triangle_{k+1}$ witnessing a win for player 1 (in our hypothetical game of Mines$_{\mathcal{R}_{1}(k+1, k-1)}(k+1,k-1)$. Note that not all $\triangle_{k}'s$ in $W$ are played by player 2 (in our original game) because otherwise the game would not have ended in a draw. So, without loss of generality, we can assume that there is at least one $W'\in\triangle_{k}(W)$ played by player 1 (in the original game). However, this is a contradiction because $W'\cup\{z_{1},z_{2},\dots z_{k+1}\}$ then witnesses a win for player 1 (in our original game). If instead player 2 wins, we can let $\bar{Z}$ be some element of $\triangle_{3}(\bar{Y})$ and we obtain a similar contradiction.
\end{proof}

\subsection{The probabilistic method}

The work in this section follows closely from the probabilistic method described by Erd\"{o}s \cite{ProbMethod}. In our case, we apply it to a randomly played game of Mines.

\begin{theorem}\label{lower bound 1} Let $p,q,k$ and $m$ be natural numbers such that $k\le p\le q\le m$.
If ${m \brack p} \cdot 2^{-{p \brack k}} + {m \brack q} \cdot 2^{-{q \brack k}} < 1$ then $\mathcal{R}_{1}(p,q,k)> m.$
\end{theorem}
\begin{proof}
	Let $p,q,k$ and $m$ be given. To use the probabilistic method, we consider two players playing a game of Mines$_{m}(p,q,k)$ game with a fair coin. The players pick a position and flip the coin to see if they will play that triangular set or skip their turn. The game is then played to completion using these random moves.
	
	Consider the random variables $X$ and $Y$ where $X$ counts the number of winning $\triangle_p$'s colored by player 1 and $Y$ the number of winning $\triangle_q$'s colored by player 2.  By Theorem \ref{counting}, the probability that a randomly chosen $\triangle_{p}$ from the game board witnesses a win for player one  is $2^{-{p \brack k}}$ and for a randomly chosen $\triangle_q$ the probability that it witnesses a win for player two is $2^{-{q \brack k}}$. By Theorem \ref{counting}, there are exactly ${m \brack p}$ possible $\triangle_{p}$'s on the game board and ${m \brack q}$ possible $\triangle_{q}$'s. Thus, ${m \brack p}\cdot2^{1 - {p \brack k}}=E[X]$ and ${m \brack q}\cdot2^{1 - {q \brack k}}=E[Y]$. Thus $E[X+Y]=E[X]+E[Y]={m \brack p} \cdot 2^{-{p \brack k}} + {m \brack q} \cdot 2^{-{q \brack k}} < 1$. Since the expected value of $X+Y$ is less than one then there exists some game of Mines$_{m}(p,q,k)$ that ends in a draw. Therefore $m<\mathcal{R}_{1}(p,q,k)$.
	\end{proof}
The previous Theorem can be used to find lower bounds for small values of $p,q$ and $k$ by searching for the largest value of $m$ that satisfies the inequality. A summary of these values for small $p,q$ and $k$ can be found in Figure \ref{TRN}. The next Theorem provides asymptotic estimates for large values of $p,q$ and $k$ with $p=q=n$ which don't require the computation of ${n \brack k}$ nor ${m \brack n}$ for any $m$.

	\begin{theorem}\label{lower bound 2}
		For all natural numbers $n$ and $k$ with $n\le k$,
		\[m > (2 \pi T_n)^{\frac{1}{4T_n}} \cdot \sqrt{ \frac{2T_n}{e}} \cdot 2^{\frac{n^k - k^k}{2k^kT_n}} - 1\]
		
	\end{theorem}
	
	\begin{proof}
		
		Let $m=\mathcal{R}_{1}(n,k)$. By the previous Theorem we must have ${m \brack n} 2^{-{n \brack k} + 1} \ge 1.$ Now we also know that 
		$${m \brack n} \le {T_m \choose T_n} $$ 
		This is because if we are counting $\triangle_n$'s (which has $T_n$ points) in a $\triangle_m$ (which has $T_m$ points). Therefore we are counting some $T_n$ sized subset of $T_m$ points with certain properties. We know that ${m \choose n} \leq \frac{m^n}{n!}$
		Substituting in these inequalities gives, 
		\[ \frac{T_m^{T_n}}{T_n!} \cdot 2^{-{n \brack k}+1} \geq 1\]
		\[ \implies \frac{\left( m(m+1) \right)^{T_n}}{T_n!} \cdot 2^{-{n \brack k} +1 - T_n} \geq 1 \]
		Then by Sterling's formula that $n! > \sqrt{2 \pi n} \cdot \left( \frac{n}{e} \right)^n$ we have
		\[ \frac{\left( m(m+1) \right)^{T_n}}{ \sqrt{2 \pi T_n } \cdot T_n^{T_n} \cdot e^{-T_n}} \cdot 2^{-{n \brack k} +1 - T_n} \geq 1 \]
		\[ \implies \frac{ (m+1)^{2T_n}}{ \sqrt{2 \pi T_n } \cdot T_n^{T_n} \cdot e^{-T_n}} \cdot 2^{-{n \brack k} +1 - T_n}  \geq 1 \]
		\[ \implies (m+1)^{2T_n} \geq \sqrt{2 \pi T_n} \cdot \left( \frac{T_n}{e} \right)^{T_n} \cdot 2^{{n \brack k} + T_n -1}  \]
		
		From Corollary \ref{brack sums}, $ {n \brack k} = \sum\limits_{0<i_1<...<i_k\leq{n}} \Bigg( \prod\limits_{j=1}^k {i_j \choose j} \Bigg).$ The smallest product in this summation is $ {1 \choose 1}{2 \choose 2}\cdots{n \choose n} = 1$ since there are ${n \choose k}$ terms in the sum $ {n \brack k} \geq  {n \choose k} \geq \left(\frac{n}{k}\right)^k $. So 
		\[ (m+1)^{2T_n} \geq \sqrt{2 \pi T_n} \cdot \left( \frac{T_n}{e} \right)^{T_n} \cdot 2^{\left( \frac{n}{k} \right)^k + T_n -1} \]
		Isolating $m$ in the previous inequality and simplifying gives the result.

	\end{proof}

\section{Conclusion}\label{Section6}
 Recall that the existence of triangular Ramsey numbers follows from the work of Dobrinen and Todorcevic in \cite{DandT}. Our primary purpose for introducing the game of Mines was to provide a simplified presentation of the finite-dimensional Ramsey theory of the infinite-dimensional topological Ramsey space $\mathcal{R}_{1}$ introduced and studied by Dobrinen and Todorcevic in \cite{DandT}. The next table summarizes the main results for small values of $p,q$ and $k$. Note that $\mathcal{R}_{1}(p,q,k)=\mathcal{R}_{1}(q,p,k)$ so we only give values for $p\le q$.

\begin{figure}[h]
		\begin{center}
			$$\begin{array}{|c|c|c|c|c|c|c|} \hline
			p&q&k&$Lower Bound$&\mathcal{R}_{1}(p,q,k) &$Upper Bound$&$Result$ \\ \hline 
			$1$&$2$&$1$&-&$2$&-& $Thm. $\ref{Delta1Thm} \\ \hline 
			$2$&$2$&$1$&-&$3$&-&$Thm. $\ref{Delta1Thm}\\ \hline
			$2$&$3$&$1$&-&$4$&-&$Thm. $\ref{Delta1Thm} \\ \hline
			$3$&$3$&1&-&5&-&$Thm. $\ref{Delta1Thm}\\ \hline \hline
			$2$&$2$&$2$&-&2&-& $trivial$ \\ \hline
			$2$&$3$&$2$&-&3& - & $trivial$ \\ \hline
			$3$&$3$&$2$&6&?& R^{15}(3)& $Thm. $\ref{MTheorem2} \\ \hline
			$3$&$4$&$2$&6&?& M_{7,2}& $Thm. $\ref{MTheorem} \& $Thm. $\ref{lower bound 1} \\ \hline
			$4$&$4$&$2$&25&?& M_{7,2} & $Thm. $\ref{MTheorem} \& $Thm. $\ref{lower bound 1} \\ \hline \hline
			$3$&$3$&$3$&-&3& - &$trivial$\\ \hline
			$3$&$4$&$3$&-& 4 & - & $trivial$ \\ \hline
			$4$&$4$&$3$&20&?& R^{{M_{7,2} \brack 2}}(4,3) & $Thm. $\ref{MTheorem2}\& $Thm. $\ref{lower bound 1}\\ \hline
			$4$&$5$&$3$&20&?& M_{9,3}& $Thm. $\ref{MTheorem} \& $Thm. $\ref{lower bound 1} \\ \hline
			$5$&$5$&$3$&9.39\times 10^7&?& M_{9,3}& $Thm. $\ref{MTheorem} \& $Thm. $\ref{lower bound 1} \\ \hline\hline
			$4$&$4$&$4$&-& 4 & - & $trivial$ \\ \hline
			$4$&$5$&$4$&-& 5 & - & $trivial$ \\ \hline
			$5$&$5$&$4$&3425&?& R^{{M_{9,3} \brack 3}}(5,4) &$Thm. $\ref{MTheorem2} \& $Thm. $\ref{lower bound 1} \\ \hline\hline
			$30$&$30$&$20$& 221 &?& M_{59,20} & $Thm. $\ref{MTheorem} \& $Thm. $\ref{lower bound 2} \\ \hline
			$35$&$35$&$20$& 4.70\time 10^{18} &?& M_{69,20} & $Thm. $\ref{MTheorem} \& $Thm. $\ref{lower bound 2} \\ \hline
			$40$&$40$&$20$& 7.29\times 10^{193} &?& M_{79,20} & $Thm. $\ref{MTheorem} \& $Thm. $\ref{lower bound 2} \\ \hline
			\end{array}$$
		\end{center}
		\caption{Table of some triangular Ramsey numbers.}
		\label{TRN}
	\end{figure}
	
These upper and lower bounds have applications to characterizing the Dedekind cuts in nonstandard models of arithmetic that arise from ultrafilter mapping that are associated to the space $\mathcal{R}_{1}$ introduced by Dobrinen and Todorcevic in \cite{DandT}. In the paper \cite{BlassCut}, Blass uses upper and lower bounds for Ramsey numbers to characterize, under the continuum hypothesis, the Dedekind cuts that can be associated to ultrafilter mappings from Ramsey and weakly-Ramsey ultrafilters. Trujillo in \cite{TrujilloCut}, characterizes the Dedekind cuts that can be associated to ultrafilter mappings among ultrafilters within the Tukey-type of a Ramsey for $\mathcal{R}_{1}$ ultrafilter. This motivates the following open problem:
\begin{problem}
Use upper and lower bounds for triangular Ramsey numbers to characterize the Dedekind cuts that can be associated to ultrafilter mappings from Ramsey for $\mathcal{R}_{1}$ ultrafilters.
\end{problem}	 In the follow up paper \cite{DandT2}, Dobrinen and Todorcevic introduce a hierarchy of spaces $\mathcal{R}_{\alpha}$ for $\alpha<\omega$ that extend the space $\mathcal{R}_{1}$. 
\begin{problem}
Introduce a combinatorial game, similar to Mines, that provides a simplified presentation of the finite-dimensional Ramsey theory of the infinite-dimensional topological Ramsey spaces $\mathcal{R}_{\alpha}$ for $\alpha<\omega_{1}$ defined and studied by Dobrinen and Todorcevic in \cite{DandT2}. Then find upper and lower bounds for Ramsey numbers based on these games.
\end{problem}
Upper and lower bounds for the Ramsey numbers associated to the spaces $\mathcal{R}_{\alpha}$, $\alpha>1$, also have similar applications to characterizing the Dedekind cuts that can be associated to ultrafilter mappings from Ramsey for $\mathcal{R}_{\alpha}$ ultrafilters. In addition to these open problems there are problems still open related to playing the game of Mines.

 Theorem 4 provides an explicit description of a winning strategy for Mines$_{3}$. The proof of Zermelo's theorem in \cite{Yurii} can be adapted to show that either player one or player two must have a winning strategy for any game of Mines$_{m}(n,k)$ where $m\ge \mathcal{R}_{1}(n,k)$. A standard strategy stealing argument can then be used to show that player one must have a winning strategy for the game. However, the proof of Zermelo's theorem does not provide for an explicit description of how player one should play to win the game.  In the game of Mines$_{5}(3,1)$, Player one has the opening move allowing them to guarantee two of the three corners (assuming this is still the optimal strategy with larger game boards) which could give them the win. However, player two has the ability to react to player one's moves and possibly prevent them from forming a $\triangle_3$. In addition, some game boards will have an odd number of positions giving player one an additional position over player two; perhaps this gives them an even bigger advantage. Regardless, giving an explicit description of the winning strategy for a game board of size $\triangle_m$ seems like a difficult problem.

\begin{problem}
Find an explicit description of the winning strategy for player one in a game of Mines$_{m}(n,k)$ where $m\ge \mathcal{R}_{1}(n,k)$. In particular, describe the winning strategy for player one in a game of Mines$_{5}(3,1)$.
\end{problem}

The complexity of Mines increases as the number of players is increased from 2 players to $k$ players. In this variation, we have a different set of Ramsey numbers. It is clear that as the number of players increases the size of the associated triangular Ramsey numbers also increase.
\begin{problem}
 How does the game of Mines change when adding more players? In particular, find upper and lower bounds for triangular Ramsey numbers for variations of Mines with more than two players.
\end{problem}

In the off-diagonal case of the game Mines$_{m}(p,q,1)$, if $p<q$, then player one has an advantage as they are now constructing a smaller triangular set. Player two can be given an advantage in this game by allowing them to play more than one $\triangle_{1}$ on each turn. The question then becomes, how many $\triangle_1$s should player two be allowed to color per turn such that the game is fair? Note that, if $m=p+q-1$ and we let player two play $q$ $\triangle_{1}$s per turn then they clearly have a winning strategy provided that $p>1$. 

\begin{problem}
Suppose $p$ and $q$ are given with $1<p<q$. What is the smallest number of $\triangle_{1}$s player two can be allowed to color per turn such that player two has the winning strategy?
\end{problem}

In addition to simplifying the finite Ramsey theory of the Ramsey space $\mathcal{R}_{1}$, a secondary goal was to introduce a game that was simple enough to played by young children. In this way, the game can be played with the intention of developing  logical skill. In the paper \cite{haggard1977game}, the authors give evidence that the game of Tri can be used to develop logical thinking skills in young children. Our final problem will be of interest mainly to researchers in math education.

\begin{problem}
Give concrete evidence that the game Mines can be played by young children and used to develop their visual disembedding skills.
\end{problem}

------------

\end{document}